\documentclass[11pt,a4paper]{article}

\usepackage{amssymb,amsmath,amsthm,graphicx}
\usepackage{intcalc}
\usepackage{todonotes,enumerate}
\usepackage{authblk,cite}
\usepackage[ruled,vlined,linesnumbered]{algorithm2e}
\usepackage{etoolbox}
\providetoggle{long}
\usepackage[none]{hyphenat}
\settoggle{long}{true}
\usepackage{amssymb,graphicx}
\usepackage[plain]{fullpage}
\usepackage[numbers]{natbib}
\usepackage{xcolor}
\usepackage{bbm,hyperref}
\hypersetup{
	colorlinks=true,
    linkcolor={red!50!black},
    citecolor={blue!50!black},
    urlcolor={blue!80!black},
	bookmarksopen=true,
	bookmarksnumbered,
	bookmarksopenlevel=2,
	bookmarksdepth=3
}
\usepackage{cleveref}[capitalise,2011/12/24]

\newtheorem{theorem}{Theorem}[section]
\newtheorem{corollary}[theorem]{Corollary}
\crefname{corollary}{Corollary}{Corollaries}

\newtheorem{observation}[theorem]{Observation}

\newtheorem{question}{Question}

\newtheorem{proposition}[theorem]{Proposition}
\newtheorem{lemma}[theorem]{Lemma}
\crefname{lemma}{Lemma}{Lemmas}
\theoremstyle{definition}

\newtheorem{remark}[theorem]{Remark}

\newcommand{\XK}{K'}
\newcommand{\XI}{I'}

\graphicspath{{./figures/}}

\usepackage{comment}

\title{The Upper Clique Transversal Problem:\\ Hardness and Algorithms\thanks{A preliminary version of this work appeared in the proceedings of the 49th International Workshop on Graph-Theoretic Concepts in Computer Science (WG 2023)~\cite{MR4657732}.}}

\author{Martin Milani\v c\thanks{Corresponding author.}\\
\small FAMNIT and IAM, University of Primorska, Koper, Slovenia\\
\small \texttt{martin.milanic@upr.si}
\and
Yushi Uno\\
\small Graduate School of Informatics,
\small Osaka Metropolitan University,
\small Sakai, Osaka, Japan\\
\small \texttt{yushi.uno@omu.ac.jp}
}

\date{}

\begin{document}
\maketitle
\begin{abstract}
\begin{sloppypar}

A \emph{clique transversal} in a graph is a set of vertices intersecting all maximal cliques.
The problem of determining the minimum size of a clique transversal has received considerable attention in the literature.
In this paper, we initiate the study of the ``upper'' variant of this parameter, the \emph{upper clique transversal number}, defined as the maximum size of a minimal clique transversal.
We investigate this parameter from the algorithmic and complexity points of view, with a focus on various graph classes.
We show that the corresponding decision problem is \hbox{{\sf NP}-complete} in the classes of chordal graphs, chordal bipartite graphs, cubic planar bipartite graphs, and line graphs of bipartite graphs, but solvable in linear time in the classes of split graphs, proper interval graphs, and cographs, and in polynomial time for graphs of bounded cliquewidth.


\end{sloppypar}
\end{abstract}

\section{Introduction}

A set of vertices of a graph $G$ that meets all maximal cliques of $G$ is called a \emph{clique transversal} in $G$.
Clique transversals in graphs have been studied by Payan in 1979~\cite{MR539710}, by Andreae, Schughart, and Tuza in 1991~\cite{MR1099264}, by Erd\H{o}s, Gallai, and Tuza in 1992~\cite{MR1189850}, and also extensively researched in the more recent literature (see, e.g.,~\cite{MR1201987,MR1375117,MR1413638,MR1423977,MR1737764,MR4213405,MR4264990,MR3875141,MR3350239,MR3325542,MR3131902,MR2203202}).
What most of these works have in common is that they focus on questions regarding the \textit{clique transversal number} of a graph, that is, the minimum size of a clique transversal of the graph.
For example, Chang, Farber, and Tuza showed in~\cite{MR1201987} that computing the clique transversal number for split graphs is {\sf NP}-hard, and Guruswami and Pandu Rangan showed in~\cite{MR1737764} that the problem is {\sf NP}-hard for cocomparability, planar, line, and total graphs, and solvable in polynomial time for Helly circular-arc graphs, strongly chordal graphs, chordal graphs of bounded clique size, and cographs.

In this paper, we initiate the study of the ``upper'' version of this graph invariant, the \textit{upper clique transversal number}, denoted by $\tau_c^+(G)$ and defined as the maximum size of a minimal clique transversal, where a clique transversal in a graph $G$ is said to be \emph{minimal} if it does not contain any other clique transversal.
The corresponding decision problem is defined as follows.

\begin{center}
\fbox{\parbox{.97\linewidth}{\noindent
\textsc{Upper Clique Transversal (UCT)}\\[.8ex]
\begin{tabular*}{.93\textwidth}{rl}
{\em Input:} & A graph $G$ and an integer $k$.\\
{\em Question:} & Does $G$ contain a minimal clique transversal $S$ such that $|S|\ge k$?
\end{tabular*}
}}
\end{center}

Our study contributes to the literature on upper variants of graph minimization problems, which already includes the upper vertex cover (also known as maximum minimal vertex cover; see~\cite{MR2772557,MR3717814,MR3399960}), upper feedback vertex set (also known as maximum minimal feedback vertex set; see~\cite{MR4322274,MR4941917}), upper edge cover (see~\cite{MR4075059}), upper domination (see~\cite{MR1088560,MR3807977,MR3767516}), and upper edge domination (see~\cite{MR4266848}).

\subsection*{Our results}
We provide a first set of results on the algorithmic complexity of \textsc{Upper Clique Transversal}.
Since clique transversals have been mostly studied in the class of chordal graphs and related classes, we also find it natural to first focus on this interesting graph class and its subclasses.
In this respect, we provide an {\sf NP}-completeness result as well as two very different linear-time algorithms.
We show that UCT is {\sf NP}-complete in the class of chordal graphs, but solvable in linear time in the classes of split graphs and proper interval graphs.
Note that the result for split graphs is in contrast with the aforementioned {\sf NP}-hardness result for computing the clique transversal number in the same class of graphs~\cite{MR1201987}.
In addition, we provide {\sf NP}-completeness proofs for three more subclasses of the class of namely for chordal bipartite graphs, cubic planar bipartite graphs, and line graphs of bipartite graphs.
We also show that UCT is solvable in linear time in the class of cographs and in polynomial time in any class of graphs with bounded cliquewidth.

The diagram in \Cref{fig:graphclass} summarizes the relationships between various graph classes studied in this paper and indicates some boundaries of tractability of the UCT problem.
We define those graph classes in the corresponding later sections in the paper.
For further background and references on graph classes, we refer to~\cite{MR1686154}.

\begin{figure}[htb]
 	\centering
 	\includegraphics[width=0.85\textwidth]{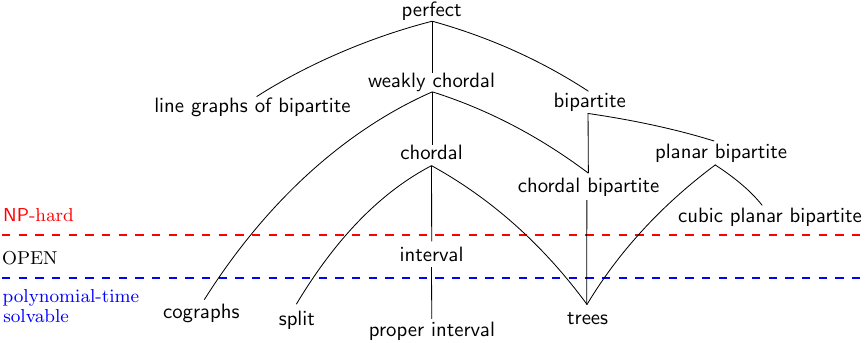}
 	\caption{The complexity of UCT in various graph classes studied in this paper.\protect\footnotemark} 	\label{fig:graphclass}
\end{figure}
\footnotetext{Recently, a polynomial-time algorithm for the \textsc{Upper Clique Transversal} problem for the class of interval graphs was given by Jaffke, Lima and Nikabadi~\cite{JLN2026}.}

\subsection*{Our approach}

We identify and make use of a number of connections between the upper clique transversal number and other graph parameters.
For example, several of our {\sf NP}-completeness proofs are based on the fact that for triangle-free graphs without isolated vertices, minimal clique transversals are exactly the minimal vertex covers, and they are closely related with minimal edge covers via the line graph operator.
In particular, if $G$ is a triangle-free graph without isolated vertices, then the upper clique transversal number of $G$ equals the upper vertex cover number of $G$, that is, the maximum size of a minimal vertex cover.

\begin{sloppypar}
Since the upper vertex cover number of a graph $G$ plus the independent domination number of $G$ equals the order of $G$, there is also a connection with the independent dominating set problem.
Let us note that, along with a linear-time algorithm for computing a minimum independent dominating set in a tree~\cite{MR0485473}, the above observations suffice to justify the polynomial-time solvability of the upper clique transversal problem on trees, as indicated in \Cref{fig:graphclass}.
They are also instrumental in the {\sf NP}-completeness proofs for the classes of chordal bipartite graphs and cubic planar bipartite graphs.
\end{sloppypar}

The {\sf NP}-completeness proofs for the classes of chordal graphs and line graphs of bipartite graphs are based on a reduction from {\sc Spanning Star Forest}, the problem of computing a spanning subgraph with as many edges as possible that consists of disjoint stars; this problem, in turn, is known to be closely related to the dominating set problem.

The linear-time algorithm for computing the upper clique transversal number of proper interval graphs relies on a linear-time algorithm for the maximum induced matching problem in bipartite permutation graphs due to Chang~\cite{MR2024264}.
More precisely, we prove that the upper clique transversal number of a given graph cannot exceed the maximum size of an induced matching of a derived bipartite graph, the \emph{vertex-clique incidence graph}, and show, using new insights on the properties of the matching computed by Chang's algorithm, that for proper interval graphs, the two quantities are the same.

Our approach in the case of split graphs is based on a characterization of minimal clique transversals of split graphs.
A clique transversal that is an independent set is also called a \emph{strong independent set} (or \emph{strong stable set}; see~\cite{MR4273625} for a survey).
It is not difficult to see that every strong independent set is a minimal clique transversal.
We show that every split graph has a maximum minimal clique transversal that is independent (and hence, a strong independent set).
In particular, this result implies that within the class of split graphs, the independence number is an upper bound for the upper clique transversal number.

\begin{sloppypar}
The linear-time algorithm for UCT in the class of cographs is based on the recursive structure of cographs and the fact that, within the class of cographs, the upper clique transversal number coincides with the independence number.
Finally, we complement our polynomial results by observing that UCT can be formulated in {\sf MSO$_1$} logic, which immediately leads to a polynomial-time algorithm for computing the upper clique transversal number of graphs with bounded cliquewidth, by applying a metatheorem due to Courcelle, Makowsky, and Rotics~\cite{MR1739644}.
\end{sloppypar}

\subsection*{Structure of the paper}

In \Cref{sec:prelim} we introduce the relevant graph theoretic background.
Hardness results are presented in \Cref{sec:hardness}.
Linear-time algorithms for UCT in the classes of split graphs, proper interval graphs, and cographs are developed in \Cref{sec:split-graphs,sec:PIGs,sec:cographs}, respectively.
In \Cref{sec:bounded-cliquewidth}, we show that UCT can be solved in polynomial time in any class of graphs with bounded cliquewidth.
We conclude the paper with a number of open questions in \Cref{sec:conclusion}.

\medskip
\begin{sloppypar}
The paper contains detailed proofs of all the results presented in the conference version~\cite{MR4657732}, as well as new results, namely the \textsf{NP}-completeness of UCT in the class of cubic planar bipartite graphs and polynomial-time algorithms for the classes of cographs and graphs with bounded cliquewidth.
It also contains a more extensive concluding discussion, including several open questions and remarks on the complexity of UCT in relation to width parameters other than cliquewidth, namely for graph classes having bounded tree-independence number, mim-width, sim-width, or twin-width.
\end{sloppypar}

\section{Preliminaries}\label{sec:prelim}

Throughout the paper, graphs are assumed to be finite, simple, and undirected. 
We use standard graph theory terminology, following West~\cite{MR1367739}.
A graph $G$ with vertex set $V$ and edge set $E$ is often denoted by $G=(V, E)$; we write $V(G)$ and $E(G)$ for $V$ and $E$, respectively.
The set of vertices adjacent to a vertex $v\in V$ is the \emph{neighborhood} of $v$, denoted $N(v)$; its cardinality is the \emph{degree} of $v$, denoted $\deg(v)$.
A graph $G$ is \emph{cubic} if every vertex has degree~$3$.
The \emph{closed neighborhood} is the set $N[v]$, defined as $N(v)\cup \{v\}$.
An \emph{independent set} in a graph is a set of pairwise non-adjacent vertices; a \emph{clique} is a set of pairwise adjacent vertices.
An independent set (resp.,~clique) in a graph $G$ is \emph{maximal} if it is not contained in any other independent set (resp.,~clique).
A \emph{clique transversal} in a graph is a subset of vertices that intersects all the maximal cliques of the graph.
A \textit{dominating set} in a graph $G = (V,E)$ is a set $S$ of vertices such that every vertex not in $S$ has a neighbor in $S$.
An \textit{independent dominating set} is a dominating set that is also an independent set.
The \textit{(independent) domination number} of a graph $G$ is the minimum size of an (independent) dominating set in $G$.
Note that a set $S$ of vertices in a graph $G$ is an independent dominating set if and only if $S$ is a maximal independent set.
In particular, the independent domination number of a graph is a well-defined invariant leading to a decision problem called {\sc Independent Dominating Set}.

The \emph{clique number} of $G$ is denoted by $\omega(G)$ and defined as the maximum size of a clique in $G$.
An \emph{upper clique transversal} of a graph $G$ is a minimal clique transversal of maximum size.
The \emph{upper clique transversal number} of a graph $G$ is denoted by $\tau_c^+(G)$ and defined as the maximum size of a minimal clique transversal in $G$.
A \emph{vertex cover} in $G$ is a set $S\subseteq V(G)$ such that every edge $e\in E(G)$ has at least one endpoint in $S$.
A vertex cover in $G$ is \emph{minimal} if it does not contain any other vertex cover.
These notions are illustrated in~\Cref{fig:uct}.
Note that if $G$ is a triangle-free graph without isolated vertices, then the maximal cliques of $G$ are exactly its edges, and hence the clique transversals of $G$ coincide with its vertex covers.

\begin{figure}[htb]
 	\centering
 	\includegraphics[width=.90\textwidth]{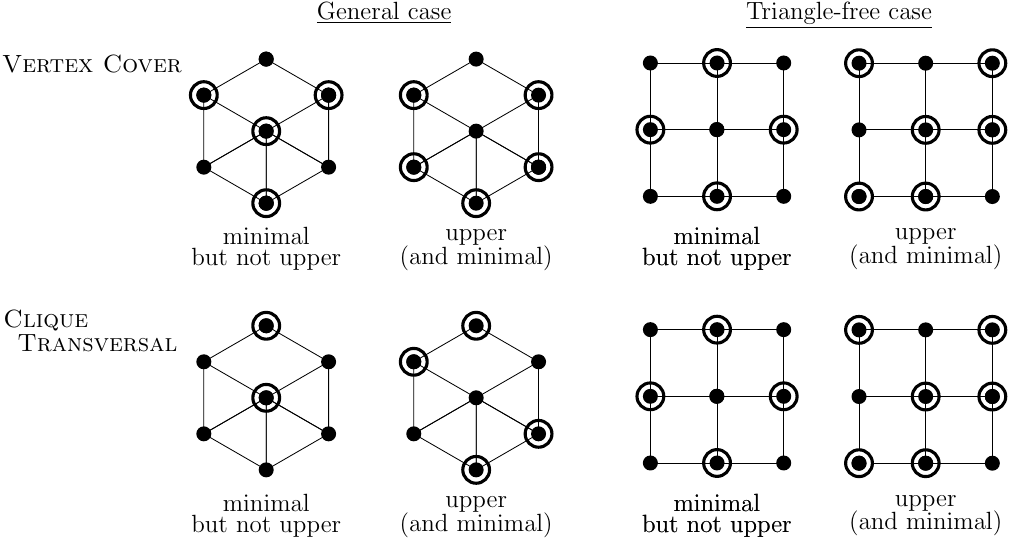}
 	\caption{Upper clique transversal and related notions.}
 	\label{fig:uct}
\end{figure}

\section{Intractability of UCT for some graph classes}\label{sec:hardness}

In this section we prove that {\sc Upper Clique Transversal} is {\sf NP}-complete in the classes of chordal graphs, chordal bipartite graphs, cubic planar bipartite graphs, and line graphs of bipartite graphs.
First, let us note that for the class of all graphs, we do not know whether the problem is in {\sf NP}.
If $S$ is a minimal clique transversal in $G$ such that $|S|\ge k$, then a natural way to verify this fact would be to certify separately that $S$ is a clique transversal and that it is a minimal one.
Assuming that $S$ is a clique transversal, one can certify minimality simply by exhibiting for each vertex $u\in S$ a maximal clique $C$ in $G$ such that $C\cap S = \{u\}$.
However, unless {\sf P} = {\sf NP}, we cannot verify the fact that $S$ is a clique transversal in polynomial time.
This follows from a result of Zang~\cite{MR1344757}, showing that it is {\sf co-NP}-complete to check, given a weakly chordal graph $G$ and an independent set $S$, whether $S$ is a clique transversal in $G$.
A graph $G$ is \textit{weakly chordal} if neither $G$ nor its complement contain an induced cycle of length at least five.

We do not know whether {\sc Upper Clique Transversal} is in {\sf NP} when restricted to the class of weakly chordal graphs.
However, for their subclasses chordal graphs and chordal bipartite graphs, membership of UCT in {\sf NP} is a consequence of the following proposition and the fact that graphs in these classes have only polynomially many maximal cliques.

\begin{proposition}\label{prop:poly-many-maximal-cliques}
Let $\mathcal{G}$ be a graph class such that every graph $G\in \mathcal{G}$ has at most polynomially many maximal cliques.
Then, \textsc{Upper Clique Transversal} is in {\sf NP} for graphs in $\mathcal{G}$.
\end{proposition}

\begin{proof}
Given a graph $G\in \mathcal{G}$ and an integer $k$, a polynomially verifiable certificate
of the existence of a minimal clique transversal $S$ in $G$ such that $|S|\ge k$ is any such set $S$.
Indeed, in this case we can enumerate all maximal cliques of $G$ in polynomial time by using any of
the output-polynomial algorithms for this task~(e.g.,~\cite{MR2159537}). In particular, we
can verify that $S$ is a clique transversal of $G$ in polynomial time. We can also verify minimality in polynomial time, by determining whether for each vertex $u\in S$ there exists a maximal clique $C$ in $G$ such that $C\cap S = \{u\}$.
\end{proof}

A \textit{star} is a graph that has a vertex that is adjacent to all other vertices, and there are no other edges.
A \textit{spanning star forest} in a graph $G = (V,E)$ is a spanning subgraph $(V,F)$ consisting of vertex-disjoint stars.
Note that a spanning star forest may contain an isolated vertex.
Some of our hardness results will make use of a reduction from {\sc Spanning Star Forest}, the problem that takes as input a graph $G$ and an integer $\ell$, and the task is to determine whether $G$ contains a spanning star forest $(V,F)$ such that $|F|\ge \ell$.

\begin{theorem}\label{thm:ssf-bip-delta-two}
\textsc{Spanning Star Forest} is {\sf NP}-complete in the class of bipartite graphs with minimum degree at least $2$.
\end{theorem}

\begin{proof}
Membership in {\sf NP} is clear.
\textsc{Spanning Star Forest} is {\sf NP}-complete due to its close relationship with {\sc Dominating Set}, the problem that takes as input a graph $G$ and an integer $k$, and the task is to determine whether $G$ contains a dominating set $S$ such that $|S|\le k$.
The connection between the spanning star forests and dominating sets is as follows: a graph $G$ has a dominating set with at most $k$ vertices if and only if $G$ has a spanning star forest with at least $|V|-k$ edges (see~\cite{MR1887943,MR2421073}).
{\sc Dominating Set} is known to be {\sf NP}-complete in the class of bipartite graphs (see, e.g.,~\cite{MR761623}) and even in the class of chordal bipartite graphs, as shown by M\"{u}ller and Brandst\"{a}dt~\cite{MR918093}.
The graphs constructed in the  {\sf NP}-hardness reduction from~\cite{MR918093} do not contain any vertices of degree zero or one, hence the claimed result follows.
\end{proof}

We present the hardness results in increasing order of difficulty of the proofs, starting with two subclasses of the class of bipartite graphs.
A \emph{chordal bipartite} graph is a bipartite graph in which all induced cycles are of length four.

\begin{theorem}\label{thm:NP-c-chordal-bip}
\textsc{Upper Clique Transversal} is {\sf NP}-complete in the classes of chordal bipartite graphs and cubic planar bipartite graphs.
\end{theorem}

\begin{proof}
Proposition~\ref{prop:poly-many-maximal-cliques} implies that UCT is in {\sf NP} when restricted to any class of bipartite graphs.
Let $\mathcal G$ be either the classes of chordal bipartite graphs or cubic planar bipartite graphs.
To prove {\sf NP}-hardness, we make a reduction from {\sc Independent Dominating Set} in $\mathcal G$, the problem that takes as input a  graph $G\in \mathcal G$ and an integer $k$, and the task is to determine whether $G$ contains a maximal independent set $I$ such that $|I|\le \ell$.
This problem is {\sf NP}-complete, as proved by Damaschke, M\"uller, and Kratsch~\cite{MR1081460} for the class of chordal bipartite graphs, and by Loverov and Orlovich~\cite{MR4112850} for the class of cubic planar bipartite graphs.
We may assume without loss of generality that the input graph $G$ does not have any isolated vertices. 
Then, given a set $I\subseteq V(G)$, the following statements are equivalent:
\begin{enumerate}[(i)]
\item $I$ is a (maximal) independent set in $G$,
\item $V(G)\setminus I$ is a (minimal) vertex cover in $G$, and
\item $V(G)\setminus I$  is a (minimal) clique transversal in $G$.
\end{enumerate}
Statements (i) and (ii) are equivalent for any graph, while the equivalence between statements (ii) and (iii) follows from the fact that the maximal cliques in $G$ are precisely its edges, since $G$ is triangle-free and does not have any isolated vertices.
It follows that $G$ has a maximal independent set $I$ such that $|I|\le \ell$ if and only if
$G$ has a minimal clique transversal $S$ such that $|S|\ge k$ where $k = |V(G)|-\ell$.
This completes the proof.
\end{proof}

We next consider the class of line graphs of bipartite graphs.
The \emph{line graph} of a graph $G$ is the graph $H$ with
$V(H) = E(G)$ in which two distinct vertices are adjacent if and only if they share an endpoint as edges in~$G$.

\begin{lemma}\label{lem:clique-line}
Let $G$ be a triangle-free graph with minimum degree at least $2$ and let $H$ be the line graph of $G$.
Then, the maximal cliques in $H$ are exactly the sets $E_v$ for $v\in V(G)$, where
$E_v$ is the set of edges in $G$ that are incident with $v$.
\end{lemma}

\begin{proof}
Since $G$ is triangle-free, any clique in $H$ corresponds to a set of edges in $G$ having a common endpoint.
Furthermore, since $G$ is of minimum degree at least $2$, any two sets $E_u$ and $E_v$ for $u\neq v$ are incomparable with respect to inclusion.
\end{proof}

An \emph{edge cover} of a graph $G$ is a set $F$ of edges such that every vertex of $G$ is incident with some edge of $F$.

\begin{lemma}\label{lem:clique-transversal-line}
Let $G$ be a triangle-free graph with minimum degree at least $2$ and let $H$ be the line graph of $G$.
Then, a set $F\subseteq E(G)$ is a clique transversal in $H$ if and only if $F$ is an edge cover in $G$.
Consequently, a set $F\subseteq E(G)$ is a minimal clique transversal in $H$
if and only if $F$ is a minimal edge cover in $G$.
\end{lemma}

\begin{proof}
Immediate from the definitions and~\Cref{lem:clique-line}.
\end{proof}

\begin{sloppypar}
We need one more lemma, connecting minimal edge covers with spanning star forests.
As observed by Hedetniemi~\cite{hedetniemi1983max}, the maximum size of a minimal edge cover in a graph without isolated vertices equals the maximum number of edges in a nontrivial spanning star forest (see also Manlove~\cite{MR1670163}), where a spanning star forest is said to be \emph{nontrivial} if it does not contain isolated vertices.

\begin{lemma}[Hedetniemi~\cite{hedetniemi1983max}]\label{nontrivial-spanning-star-forest-minimal-edge-cover}
Let $G=(V,E)$ be a graph without isolated vertices.
Then, the maximum size of a minimal edge cover in $G$ equals the maximum number of edges in a nontrivial spanning star forest in~$G$.
\end{lemma}

It turns out that the equivalence holds also with respect to our definition of a spanning star forest (where isolated vertices are allowed), thanks to the following observation.
\end{sloppypar}

\begin{lemma}\label{spanning-star-forests}
Let $G=(V,E)$ be a graph without isolated vertices and let $\ell$ be an integer such that $G$ contains a spanning star forest with $\ell$ edges.
Then, $G$ also contains a nontrivial spanning star forest with at least $\ell$ edges.
\end{lemma}

\begin{proof}
We may assume without loss of generality that $\ell$ is the maximum number of edges in a spanning star forest in $G$.
Let $(V,F)$ be a spanning star forest in $G$ with $\ell$ edges that, subject to this condition, minimizes the number of isolated vertices. 
Let $\ell'$ be the maximum number of edges in a nontrivial spanning star forest in $G$.
We show that $\ell'\ge \ell$ by contradiction.
Suppose that $\ell'<\ell$.
Then, there exists a vertex $v\in V$ that is not incident with any edge in $F$.
Since $G$ does not have isolated vertices, there exists an edge $e = \{v,w\}$ for some $w\in V\setminus \{v\}$. 
Let $F_1 = F \cup \{e\}$. 
Then, $(V,F_1)$ is a spanning forest in $G$, but it is not a spanning star forest, since it has more than $\ell$ edges.
This implies that the component $C_1$ of $(V,F_1)$ containing $v$ is not a star.

Let $C$ be the component of $(V,F)$ containing $w$.
Then $C = C_1-v$ and $C$ is a star, that is, there exists a vertex $z\in V(C)$ that is adjacent in $C$ to all other vertices.
Since $C_1$ is not a star, $z\neq w$.
On the other hand, vertices $w$ and $z$ are adjacent in $C$.
Let $F_2 = F_1\setminus\{\{w,z\}\}$.
Observe that $(V,F_2)$ is a spanning star forest in $G$ with $\ell$ edges.
Since $\{v,w\}\in F_2$, vertices $v$ and $w$ are not isolated vertices in $(V,F_2)$.
Vertex $z$ is not isolated in $(V,F_2)$, since otherwise $C_1$ would be a star.
Furthermore, for all $x\in V\setminus \{v,w,z\}$, vertex $x$ is isolated in $(V,F_2)$ if and only it is isolated in $(V,F)$.
This shows that $(V,F_2)$ is a spanning star forest in $G$ with $\ell$ edges that has fewer isolated vertices than $(V,F)$, a contradiction.
\end{proof}

\Cref{nontrivial-spanning-star-forest-minimal-edge-cover,spanning-star-forests} imply the following.

\begin{corollary}\label{spanning-star-forest-minimal-edge-cover}
Let $G=(V,E)$ be a graph without isolated vertices.
Then, the maximum size of a minimal edge cover in $G$ equals the maximum number of edges in a spanning star forest in~$G$.
\end{corollary}

Using \Cref{thm:ssf-bip-delta-two}, \Cref{lem:clique-transversal-line}, and \Cref{spanning-star-forest-minimal-edge-cover}, we can now prove the following.

\begin{theorem}\label{thm:line-bip}
\textsc{Upper Clique Transversal} is {\sf NP}-complete in the class of line graphs of bipartite graphs.
\end{theorem}

\begin{proof}
To argue that the problem is in {\sf NP}, we show that every line graph of a bipartite graph has at most polynomially many maximal cliques.
Let $G$ be a line graph of a bipartite graph.
Fix a bipartite graph $H$ such that $G = L(H)$.
Clearly, we may assume that $H$ has no isolated vertices.
Since $H$ is triangle-free, any clique in $G$ corresponds to a set of edges in $H$ having a common endpoint, and consequently any maximal clique in $G$ corresponds to an inclusion-maximal set of edges in $H$ having a common endpoint.
The number of such sets is bounded by the number of vertices in $H$.
Since
\[|V(H)| =  \sum_{v\in V(H)} 1 \le \sum_{v\in V(H)}\deg(v) = 2|E(H)| = 2|V(G)|\,,\]
it follows that the number of maximal cliques in $G$ is at most $2|V(G)|$.
By \Cref{prop:poly-many-maximal-cliques}, the problem is in {\sf NP}.

To prove {\sf NP}-hardness, we make a reduction from {\sc Spanning Star Forest} in the class of bipartite graphs with minimum degree at least $2$.
By \Cref{thm:ssf-bip-delta-two}, this problem is {\sf NP}-complete.
Let $G$ be a bipartite graph with minimum degree at least $2$ and let $H$ be the line graph of $G$.
By \Cref{lem:clique-transversal-line}, a set $F\subseteq E(G)$ is a minimal clique transversal in $H$ if and only if $F$ is a minimal edge cover in $G$.
Therefore, the graph $G$ contains a minimal edge cover with at least $\ell$ edges if and only if its line graph, $H$, contains a minimal clique transversal with at least $\ell$ vertices.
By \Cref{spanning-star-forest-minimal-edge-cover}, the maximum size of a minimal edge cover in $G$ equals the maximum number of edges in a spanning star forest in $G$.
Therefore, the graph $G$ contains a minimal edge cover with at least $\ell$ edges if and only if $G$ contains a spanning star forest with at least $\ell$ edges.
The claimed {\sf NP}-hardness result follows from \Cref{thm:ssf-bip-delta-two}.
\end{proof}

We now prove intractability of UCT in the class of chordal graphs.
A graph is \emph{chordal} if it does not contain any induced cycles on at least four vertices.

We first recall a known result on maximal cliques in chordal graphs.

\begin{theorem}[Berry and Pogorelcnik~\cite{MR2816655}]\label{thm:chordal-maximal-cliques}
A chordal graph $G = (V,E)$ has at most $|V|$ maximal cliques, which can be computed in time $\mathcal{O}(|V|+|E|)$.
\end{theorem}

\begin{theorem}\label{sec:NP-c-chordal}
\textsc{Upper Clique Transversal} is {\sf NP}-complete in the class of chordal graphs.
\end{theorem}

\begin{proof}
Membership in {\sf NP} follows from \Cref{thm:chordal-maximal-cliques} and \Cref{prop:poly-many-maximal-cliques}.

To prove {\sf NP}-hardness, we reduce from {\sc Spanning Star Forest}.
Let $G =  (V,E)$ and $\ell$ be an input instance of {\sc Spanning Star Forest}. We may assume without loss of generality that $G$ has an edge and that $\ell\ge 2$, since otherwise it is trivial to verify if $G$ has a spanning star forest with at least $\ell$ edges.

We construct a chordal graph $G'$ as follows. We start with a complete graph with vertex set $V$.
For each edge $e=\{u,v\}\in E$, we introduce two new vertices $x^e$ and $y^e$, and make $x^e$ adjacent to $u$, to $v$, and to $y^e$. The obtained graph is $G'$. We thus have $V(G') = V\cup X\cup Y$, where $X = \{x^e:e\in E\}$ and $Y = \{y^e:e\in E\}$.
See Fig.~\ref{fig:reduced-1} for an example. Clearly, $G'$ is chordal.
Furthermore, let $k = \ell+|E|$.

\begin{figure}[h!]
	\centering
	\includegraphics[width=0.75\textwidth]{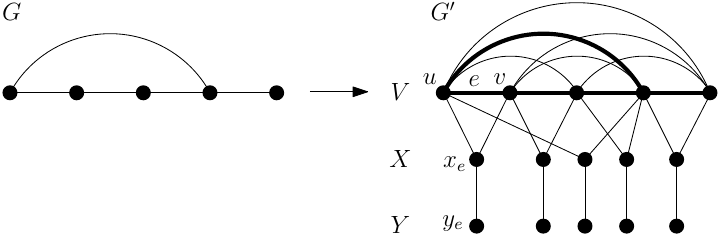}
	\caption{Transforming $G$ to $G'$.}
	\label{fig:reduced-1}
\end{figure}

To complete the proof, we show that $G$ has a spanning star forest of size at least $\ell$ if and only if $G'$ has a minimal clique transversal of size at least $k$.

First, assume that $G$ has a spanning star forest $(V,F)$ such that $|F|\ge \ell$.
Since $(V,F)$ is a spanning forest in which each component is a star, each edge of $F$ is incident with a vertex of degree one in $(V,F)$. Let $S$ be a set obtained by selecting from each edge in $F$ one vertex of degree one in $(V,F)$.
Then every edge of $F$ has one endpoint in $S$ and the other one in $V\setminus S$. In particular, $|S| = |F|\ge \ell$.
Let $S' = S\cup\{x^e:e\in E\setminus F\}\cup \{y^f:f\in F\}$. See Fig.~\ref{fig:reduced-2} for an example.

\begin{figure}[h!]
	\centering
	\includegraphics[width=0.75\textwidth]{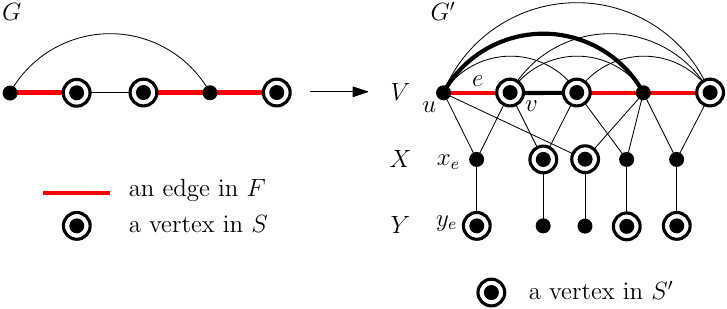}
	\caption{Transforming a spanning star forest $(V,F)$ in $G$ into a minimal clique transversal $S'$ in $G'$.}
	\label{fig:reduced-2}
\end{figure}

Clearly, the size of $S'$ is at least $\ell+|E| = k$. We claim that $S'$ is a minimal clique transversal of $G'$.
There are three kinds of maximal cliques in $G'$: the set $V$, sets of the form $\{u,v,x^e\}$ for all $e = \{u,v\}\in E$, and sets of the form $\{x^e,y^e\}$ for all $e\in E$. Since $|S| = |F|\ge \ell\ge 2$, the set $S$ is non-empty, and thus the set $S'$ intersects $V$.
Furthermore, since $S$ contains one endpoint of each edge in $F$, set $S'$ intersects all cliques of the form $\{u,v,x^f\}$ for all $f = \{u,v\}\in F$. For all $e = \{u,v\}\in E\setminus F$, set $S'$ contains vertex $x^e$ and thus also intersects the clique $\{u,v,x^e\}$, as well as the clique $\{x^e,y^e\}$. Finally, for each $f\in F$, we have that $y^f\in S'$ and hence $S'$ intersects $\{x^f,y^f\}$. Thus, $S'$ is a clique transversal of $G'$.

To argue minimality, we need to show that for every $u\in S'$ there exists a maximal clique in $G'$ missed by $S'\setminus \{u\}$.
Suppose first that $u\in V$. Then $u\in S$
and there is an edge $f\in F$ such that $u$ is an endpoint of $f$. Let $v$ be the other endpoint of $f$. Then $v\not\in S$ and thus also $v\not\in S'$. Note also that $x^f\not\in S'$. In particular, this implies that the set $\{u,v,x^f\}$ is a maximal clique of $G'$ missed by $S'\setminus \{u\}$. Next, suppose that
$u\in X$. Then $u = x^e$ for some edge $e\in E\setminus F$ and $y^e\not\in S'$, hence the set $\{x^e,y^e\}$ is a maximal clique of $G'$ missed by $S'\setminus \{u\}$. Finally, suppose that
$u\in Y$. Then $u = x^f$ for some edge $f\in F$. Then $x^f\not\in S'$, therefore the set $\{x^f,y^f\}$ is a maximal clique of $G'$ missed by $S'\setminus \{u\}$. This shows that $S'$ is a minimal clique transversal of $G'$, as claimed.

For the converse direction, let $S'$ be a minimal clique transversal of $G'$ such that $|S'|\ge k$.
First we show that $S'\cap Y\neq\emptyset$. Suppose for a contradiction that $S'\cap Y= \emptyset$. Then $X\subseteq S'$, since otherwise the maximal clique $\{x^e,y^e\}$ of $G'$ would be missed by $S'$ for every $x^e\in X\setminus S'$. Furthermore, since $V$ is a maximal clique in $G'$, there is a vertex $u\in S'$ such that $u\in V$.
Since the set $X\cup\{u\}$ is a clique transversal in $G'$, the minimality of $S'$ implies that $S' = X\cup\{u\}$.
Using the fact that $k = \ell+|E|$ and $k\le |S'| = |X|+1 = |E|+1$, we then obtain that $\ell\le 1$.
This contradicts our assumption that $\ell\ge 2$ and shows that $S'\cap Y\neq\emptyset$.

Let $S = S'\cap V$. Recall that for every edge $e\in E$ we denote by $x^e$ the unique vertex in $X$ that is adjacent in $G'$ to both endpoints of $e$. We claim that for each vertex $u\in S$ there exists a vertex $v\in V$ such that $e = \{u,v\}\in E$ and
$S'\cap \{u,v,x^e\} = \{u\}$. Let $u\in S$ and suppose for a contradiction that for all vertices $v$
such that $e = \{u,v\}\in E$ we have $S'\cap \{u,v,x^e\} \neq \{u\}$. This implies that the set $S'\setminus\{u\}$
intersects all maximal cliques in $G'$ of the form $\{u,v,x^e\}$ for some $e = \{u,v\}\in E$.
Since $S'$ is a minimal clique transversal of $G'$, we infer that the maximal clique of $G'$ missed by
$S'\setminus\{u\}$ is $V$. In particular, we have $S = S'\cap V = \{u\}$, which in turn implies that
for all vertices $v\in V$ such that $e = \{u,v\}\in E$ we have $S'\cap \{u,v,x^e\} = \{u,x^e\}$.
Since $S'\cap Y\neq\emptyset$, there exists an edge $e=\{w,z\}$ of $G$ such that $y^e\in S'$.
Then $x^e\not\in S'$, and therefore $u$ is not an endpoint of $e$. However, since, $x^e\not\in S'$ but $S'$ intersects the maximal clique $\{w,z,x^e\}$, it follows that an endpoint of $e$ belongs to $S$. This contradicts the fact that
$S = \{u\}$ and $u$ is not an endpoint of~$e$.

By the above claim, we can associate to each vertex $u\in S$ a vertex $v(u)\in V$ such that $e = \{u,v(u)\}\in E$ and $S'\cap \{u,v(u),x^e\} = \{u\}$.
For each $u\in S$, let us denote by $e(u)$ the corresponding edge $\{u,v(u)\}$, and let $F = \{e(u):u\in S\}$ (see Fig.~\ref{fig:reduced-3}). 
We next claim that every vertex in $S$ has degree~$1$ in $(V, F)$.
Suppose for a contradiction that there exists an edge $\{u,v\}\in F$ such that $u,v\in S$.
We can assume without loss of generality that $e(u) = \{u,v\}$. 
Then $S'\cap \{u, v, x^{e(u)}\} = \{u\}$ implies that $v\not\in S'$, which contradicts the fact that $v\in S$.
Thus every edge in $F$ is incident with exactly one vertex in $S$, which in turn implies that the mapping $u\mapsto e(u)$ is one-to-one and, hence, $|F| = |S|$.
It also implies that every vertex in $S$ has degree one in $(V, F)$.
Therefore, the graph $(V,F)$ is a spanning star forest of $G$.

\begin{figure}[h!]
 	\centering
\includegraphics[width=0.83\textwidth]{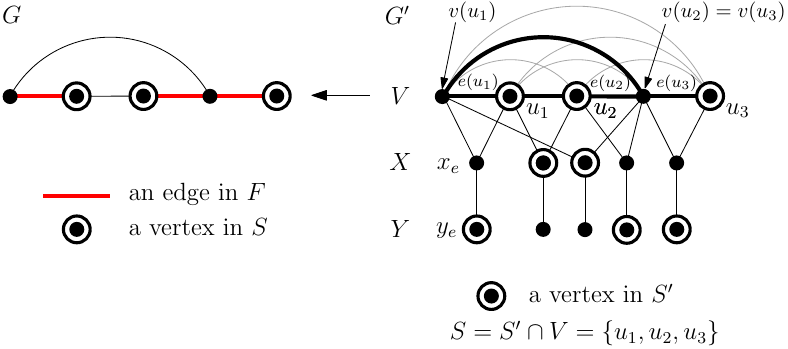}
 	\caption{Transforming a minimal clique transversal $S'$ in $G'$ into a spanning star forest $(V,F)$ in $G$.}
 	\label{fig:reduced-3}
\end{figure}

Since $S'$ is a minimal clique transversal of $G'$, for each edge $e\in E$ exactly one of $x^e$ and $y^e$ belongs to $S'$.
Therefore, $|F| = |S| = |S'|-|E|\ge k-|E| = \ell$. Thus, $G$ has a spanning star forest of size at least $\ell$.
\end{proof}

\section{A linear-time algorithm for UCT in split graphs}
\label{sec:split-graphs}

A \emph{split graph} is a graph that has a \emph{split partition}, that is, a partition of its vertex set into a clique and an independent set.
We denote a split partition of a split graph $G$ as $(K,I)$ where $K$ is a clique, $I$ is an independent set, $K\cap I = \emptyset$, and $K\cup I = V(G)$.
We may assume without loss of generality that $I$ is a maximal independent set.
Indeed, if this is not the case, then $K$ contains a vertex $v$ that has no neighbors in $I$, and $(K\setminus \{v\},I\cup \{v\})$ is a split partition of $G$ such that $I\cup\{v\}$ is a maximal independent set.
In what follows, we repeatedly use the structure of maximal cliques of split graphs. If $G$ is a split graph with a split partition $(K,I)$, then the maximal cliques of $G$ are as follows: the closed neighborhoods $N[v]$, for all $v\in I$, and the clique $K$, provided that it is a maximal clique, that is, every  vertex in $I$ has a non-neighbor in $K$.

Given a graph $G$ and a set of vertices $S\subseteq V(G)$, we denote by $N(S)$ the set of all vertices in $V(G)\setminus S$ that have a neighbor in $S$. Moreover, given a vertex $v\in S$, an \emph{$S$-private neighbor} of $v$ is any vertex $w\in N(S)$ such that $N(w)\cap S = \{v\}$.
The following proposition characterizes minimal clique transversals of split graphs.

\begin{proposition}\label{prop:MCT-split}
Let $G$ be a split graph with a split partition $(K,I)$ such that $I$ is a maximal independent set
and let $S\subseteq V(G)$.
Let $\XK = K\cap S$ and $\XI = I\cap S$.
Then, $S$ is a minimal clique transversal of $G$ if and only if the following conditions hold:
\begin{enumerate}[(i)]
  \item\label{condition:domination-2} $\XK \neq \emptyset$ if $K$ is a maximal clique.
  \item\label{condition:domination-and-minimality}  $\XI = I\setminus N(\XK)$.
  \item\label{condition:minimality-2} Every vertex in $\XK$ has a $\XK$-private neighbor in $I$.
\end{enumerate}
\end{proposition}

\begin{proof}
Assume first that $S$ is a minimal clique transversal of $G$.
We prove that $S$ satisfies each of the three conditions.
Condition~\eqref{condition:domination-2} follows from the fact that $S$ is a clique transversal.

To show~condition~\eqref{condition:domination-and-minimality}, we first show the inclusion $I\setminus S\subseteq N(\XK)$, which is equivalent to $I\setminus N(\XK) \subseteq I\setminus (I\setminus S) = I\cap S = \XI$.
Consider an arbitrary vertex $v\in I\setminus S$. Since $N[v]$ is a maximal clique in $G$ and $S$ is a clique transversal not containing $v$, set $S$ must contain a neighbor $w$ of $v$. As $N(v)\subseteq K$, we conclude that $w$ belongs to $\XK$.
The converse inclusion, $\XI \subseteq I\setminus N(\XK)$, is equivalent to the condition that there are no edges between $\XI$ and $\XK$.
Suppose for a contradiction that $G$ contains an edge $\{u,v\}$ with $u\in \XI$ and $v\in \XK$. 
Since $N[u]$ is the only maximal clique of $G$ containing $u$ and $\{u,v\}\subseteq S\cap N[u]$, it follows that $S\setminus \{u\}$ is a clique transversal of $G$, contradicting the minimality of $S$. This establishes~\eqref{condition:domination-and-minimality}.

To show condition~\eqref{condition:minimality-2}, consider an arbitrary vertex $v\in \XK$. If $\XK = \{v\}$, then any neighbor of $v$ in $I$ is a $\XK$-private neighbor of $v$, and $v$ has a neighbor in $I$ since $I$ is a maximal independent set.
Thus we may assume that $|\XK|\ge 2$.
Suppose for a contradiction that $v$ does not contain any $\XK$-private neighbor in $I$. The maximal cliques of $G$ containing $v$ are $N[w]$ for $w\in N(v)\cap I$ and possibly $K$ (if $K$ is a maximal clique). For every $w\in N(v)\cap I$, the assumption on $v$ implies that there exists a vertex $v'\in \XK\setminus\{v\}$ adjacent to $w$; hence $\{v,v'\}\subseteq S\cap N[w]$. Moreover, we had already justified that $|\XK|\ge 2$. It follows that the set $S\setminus \{v\}$ intersects all maximal cliques in $G$; this contradicts the minimality of $S$ and shows condition~\eqref{condition:minimality-2}.

Assume now that $S$ is a set of vertices satisfying conditions \eqref{condition:domination-2}--\eqref{condition:minimality-2}.
We prove that $S$ is a minimal clique transversal by verifying both conditions in the definition. Consider an arbitrary maximal clique $C$ of $G$. If $C = N[v]$ for some $v\in I$, then either $v\in S$, in which case
$v\in S\cap C$, or $v\in I\setminus S$, in which case condition~\eqref{condition:domination-and-minimality} guarantees that $v$ has a neighbor $w\in \XK$; hence $w\in S\cap C$ and $S$ intersects $C$.
If $C = K$, then $S\cap C \neq \emptyset$ by condition~\eqref{condition:domination-2}. Hence $S$ is a clique transversal.
To show minimality, suppose for a contradiction that $S$ contains a vertex $v$ such that $S\setminus\{v\}$ is also a clique transversal of $G$.
Suppose that $v\in I$. Since the set $S\setminus\{v\}$ intersects the maximal clique $N[v]$, there is a vertex $w\in (S\setminus\{v\})\cap N[v]$.
Since $w\neq v$, we have $w\in N(v)$ and hence $w\in K$.
In particular, $w\in \XK$ and thus $v\in N(\XK)\cap \XI$;  this contradicts condition \eqref{condition:domination-and-minimality}. It follows that $v\not\in I$ and hence $v\in \XK$. Condition~\eqref{condition:minimality-2} implies that $v$ has a $\XK$-private neighbor $w\in I$. 
Since $N(w)\subseteq K$ and $w$ is a $\XK$-private neighbor of $v$, we have $S \cap N(w) = N(w)\cap \XK = \{v\}$, which implies $(S\setminus \{v\})\cap N(w) = \emptyset$.
Moreover, condition \eqref{condition:domination-and-minimality} implies that $w\not\in S$; hence $(S\setminus \{v\})\cap \{w\} = \emptyset$.
It follows that
$(S\setminus \{v\})\cap N[w] = ((S\setminus \{v\})\cap N(w))\cup ((S\setminus \{v\})\cap \{w\}) = \emptyset$. Since the set $S\setminus\{v\}$ misses the maximal clique $N[w]$, it is not a clique transversal, a contradiction.
\end{proof}

\Cref{prop:MCT-split} leads to the following result about maximum minimal clique transversals in split graphs.
We denote by $\alpha(G)$ the \emph{independence number} of a graph $G$, that is, the maximum size of an independent set in $G$.

\begin{theorem}\label{thm:MMCT-split}
Let $G$ be a split graph with a split partition $(K,I)$ such that $I$ is a maximal independent set.
Then:
\begin{enumerate}
  \item If $K$ is not a maximal clique in $G$, then $I$ is a maximum minimal clique transversal in $G$; in particular, we have $\tau_c^+(G) = \alpha(G)$ in this case.
  \item If $K$ is a maximal clique in $G$, then for every vertex $v\in K$ with the smallest number of neighbors in $I$, the set $\{v\}\cup (I\setminus N(v))$ is a maximum minimal clique transversal in $G$; in particular, we have  $\tau_c^+(G) = \alpha(G)-\delta_G(I,K)+1$ in this case, where $\delta_G(I,K) = \min\{|N(v)\cap I|: v\in K\}$.
\end{enumerate}
Consequently, every split graph $G$ satisfies $\tau_c^+(G)\le \alpha(G)$.
\end{theorem}

\begin{sloppypar}
\begin{proof}
Let $S$ be a minimal clique transversal of $G$ that is of maximum possible size and, subject to this condition, contains as few vertices from $K$ as possible. Let $\XK = K\cap S$ and $\XI = I\cap S$.
If $\XK = \emptyset$, then $K$ is not a maximal clique in $G$, and we have $S = I$, implying $\tau_c^+(G) = |S| = \alpha(G)$.
Suppose now that $\XK\neq\emptyset$. We first show that $|\XK|= 1$. Suppose for a contradiction that $|\XK|\ge 2$ and let $v\in \XK$.
Let $I_v$ denote the set of $\XK$-private neighbors of $v$ in $I$ and let $S' = (S\setminus \{v\})\cup I_v$.
By \Cref{prop:MCT-split}, conditions~\eqref{condition:domination-2}--\eqref{condition:minimality-2} hold for $S$.
We claim that set $S'$ also satisfies conditions~\eqref{condition:domination-2}--\eqref{condition:minimality-2} from \Cref{prop:MCT-split}.
Since $S'\cap K = \XK\setminus\{v\}$, the assumption $|\XK|\ge 2$ implies that $S'\cap K\neq\emptyset$, thus condition~\eqref{condition:domination-2} holds for $S'$.
Since condition~\eqref{condition:domination-and-minimality} holds for $S$, we have
\[I\cap S' = I'\cup I_v = (I\setminus N(K'))\cup I_v =   I\setminus N(\XK\setminus\{v\}) = I\setminus N(S'\cap K)\,,\]
that is, condition~\eqref{condition:domination-and-minimality} holds for $S'$.
Finally, since $S'\cap K \subseteq \XK$, condition~\eqref{condition:minimality-2} for $S$ immediately implies
condition~\eqref{condition:minimality-2} for $S'$.
It follows that $S'$ is a minimal clique transversal in $G$.
Furthermore, since $v\in \XK$, vertex $v$ has a $\XK$-private neighbor in $I$, that is, the set $I_v$ is nonempty.
This implies that $|S'|\ge |S|$; in particular, $S'$ is a maximum minimal clique transversal in $G$.
However, $S'$ contains strictly fewer vertices from $K$ than $S$, contradicting the choice of $S$.
This shows that $|\XK|= 1$, as claimed.

Let $w$ be the unique vertex in $\XK$.
Since condition~\eqref{condition:domination-and-minimality} from \Cref{prop:MCT-split} holds for $S$, we have $\XI = I \setminus N(w)$.
Hence $S = \{w\}\cup (I \setminus N(w))$ and
$|S| = 1+ |I|-|N(w)\cap I|$.
Since $w\in K$, we have $|N(w)\cap I|\ge \delta_G(I,K)$ and hence $\tau_c^+(G)= |S| \le \alpha(G)-\delta_G(I,K)+1$.
On the other hand, for every vertex $z\in K$ the set
$X_z:=\{z\}\cup (I \setminus N(z))$ satisfies conditions~\eqref{condition:domination-2}--\eqref{condition:minimality-2}
from \Cref{prop:MCT-split}, which can be seen as follows.
Conditions~\eqref{condition:domination-2} and~\eqref{condition:domination-and-minimality} hold by the definition of $X_z$. Since $I$ is a maximal independent set in $G$, vertex $z\not\in I$ has a neighbor in $I$, and any neighbor of $z$ in $I$ is trivially an $(X_z\cap K)$-private neighbor of $z$. Thus condition~\eqref{condition:minimality-2} holds, too. 
It follows that $X_z$ is a minimal clique transversal in $G$.
Choosing $z$ to be a vertex in $K$ with the smallest number of neighbors in $I$, we obtain a set $X_z$ of size $\alpha(G)-\delta_G(I,K)+1$.
Thus $\tau_c^+(G)\ge |X_z| = \alpha(G)-\delta_G(I,K)+1$ and since we already proved that $\tau_c^+(G)\le \alpha(G)-\delta_G(I,K)+1$, any such $X_z$ is optimal.

Since $I$ is a maximal independent set and $K$ is nonempty, we have \hbox{$\delta_G(I,K)\ge 1$}.
Thus, $\tau_c^+(G) \le \alpha(G)$.
Suppose that $K$ is not a maximal clique in $G$. Then $I$ is a minimal clique transversal in $G$ and therefore $\tau_c^+(G)\ge |I| = \alpha(G)\ge \tau_c^+(G)$. Hence equalities must hold throughout and $I$ is a maximum minimal clique transversal.
Finally, suppose that $K$ is a maximal clique in $G$. Then every minimal clique transversal $S$ in $G$ satisfies $S\cap K\neq\emptyset$.
In this case, the above analysis shows that for every vertex $v\in K$ with the smallest number of neighbors in $I$, the set $\{v\}\cup (I\setminus N(v))$ is a maximum minimal clique transversal in $G$.
\end{proof}
\end{sloppypar}

\begin{corollary}\label{cor:linear-split}
{\sc Upper Clique Transversal} can be solved in linear time in the class of split graphs.
\end{corollary}

\begin{proof}
Let $G = (V,E)$ be a given split graph. Hammer and Simeone showed that split graphs can be characterized by their degree sequences; furthermore, that characterization yields a linear-time algorithm to compute a split partition $(K,I)$ of $G$ (see~\cite{MR637832}).
If there exists a vertex in $K$ that is not adjacent to $I$, then we move it to $I$.
Thus, in linear time we can compute a split partition $(K,I)$ of $G$ such that $I$ is a maximal independent set.
Clearly, $K$ is a maximal clique if and only if no vertex in $I$ is adjacent to all vertices of $G$.
If $K$ is not a maximal clique, then the algorithm simply returns $I$.
If $K$ is a maximal clique, then the algorithm first computes, for each vertex $v\in K$, the number of neighbors of $v$ in $I$.
For a vertex $v\in K$ with the smallest number of neighbors in $I$, the set $\{v\}\cup (I\setminus N(v))$ is returned.
\end{proof}

\begin{sloppypar}
\begin{remark}
Recall that a \emph{strong independent set} in a graph $G$ is an independent clique transversal.
If $I$ is a strong independent set in $G$, then for every vertex $v\in I$, every maximal clique $K$ containing $v$ satisfies $K\cap I = \{v\}$; it follows that every strong independent set is a minimal clique transversal. \Cref{thm:MMCT-split} implies that every split graph has a maximum minimal clique transversal that is independent, that is, it is a strong independent set.
Consequently, the problem of computing a maximum minimal clique transversal of a split graph $G$ reduces to the problem of computing a maximum strong independent set in $G$. A linear-time algorithm for a more general problem, that of computing a maximum weight strong independent set in a vertex-weighted chordal graph, was developed by Wu~\cite{MR1325490}. This gives an alternative proof of~\Cref{cor:linear-split}.
\end{remark}
\end{sloppypar}

\begin{sloppypar}
\section{A linear-time algorithm for UCT in proper interval graphs}\label{sec:PIGs}
\end{sloppypar}

A graph $G = (V,E)$ is an \emph{interval graph} if it has an \emph{interval representation}, that is, if its vertices can be put in a one-to-one correspondence with a family $(I_v:v\in V)$ of closed intervals on the real line such that two distinct vertices $u$ and $v$ are adjacent if and only if the corresponding intervals $I_u$ and $I_v$ intersect. If $G$ has a \emph{proper interval representation}, that is, an interval representation in which no interval contains another, then $G$ is said to be a \emph{proper interval graph}.

Our approach towards a linear-time algorithm for {\sc Upper Clique Transversal} in the class of proper interval graphs is based on a relation between clique transversals in $G$ and induced matchings in the so-called vertex-clique incidence graph of $G$.
This relation is valid for arbitrary graphs.

\bigskip

\subsection{UCT via induced matchings in the vertex-clique incidence graph}

Given a graph $G=(V,E)$, we denote by $B_G$ the \emph{vertex-clique incidence graph} of $G$, a bipartite graph defined as follows.
The vertex set of $B_G$ consists of two disjoint sets $X$ and $Y$ such that $X = V$ and $Y =  {\mathcal C}_G$, where $ {\mathcal C}_G$ is the set of maximal cliques in $G$.
The edge set of $B_G$ consists of all pairs $x\in X$ and $C\in {\mathcal C}_G$ that satisfy $x\in C$.	
An \emph{induced matching} in a graph $G$ is a set $M$ of pairwise disjoint edges such that the set of endpoints of edges in $M$ induces no edges other than those in $M$.
Given two disjoint sets of vertices $A$ and $B$ in a graph $G$, we say that $A$ \emph{dominates $B$ in $G$} if every vertex in $B$ has a neighbor in $A$.
Given a matching $M$ in a graph $G$ and a vertex $v\in V(G)$, we say that $v$ is \emph{$M\!$-saturated} if it is an endpoint of an edge in $M$.

Clique transversals and minimal clique transversals of a graph $G$ can be expressed in terms of the vertex-clique incidence graph as follows.

\begin{lemma}\label{lem:clique-transversals-via-B_G}
Let $G$ be a graph, let $B_G = (X,Y;E)$ be its vertex-clique incidence graph, and let $S\subseteq V(G)$.
Then:
\begin{enumerate}
\item $S$ is a clique transversal in $G$ if and only if $S$ dominates $Y$ in $B_G$.
\item $S$ is a minimal clique transversal in $G$ if and only if
$S$ dominates $Y$ in $B_G$ and there exists an induced matching $M$ in $B_G$ such that $S$ is exactly the set of
$M\!$-saturated vertices in~$X$.
\end{enumerate}
\end{lemma}

\begin{proof}
The first statement follows immediately from the definitions.

For the second statement, we prove each of the two implications separately.
Assume first that $S$ is a minimal clique transversal in $G$.
Since $S$ is a clique transversal in $G$, it dominates $Y$ in $B_G$. Furthermore, the minimality of $S$ implies that for every vertex $s\in S$ there exists
a maximal clique $y_s\in Y (= \mathcal{C}_G)$ such that $y_s\cap S = \{s\}$. Let $M = \{\{s,y_s\}\mid s\in S\}$.
We claim that $M$ is an induced matching $M$ in $B_G$ such that $S$ is exactly the set of $M\!$-saturated vertices in $X$.
First, note that each $s\in S$ is adjacent in $B_G$ to $y_s$, since $s$ belongs to the maximal clique $y_s$. Second, $M$ is a matching in $B_G$ since every $s\in S$ is by construction incident with only one edge in $M$, and if $y_{s_1} = y_{s_2}$ for two vertices $s_1,s_2\in S$, then
$\{s_1\} = y_{s_1}\cap S = y_{s_2}\cap S = \{s_2\}$ and thus $s_1 = s_2$. Third, $M$ is an induced matching in $B_G$, since otherwise $B_G$ would contain an edge of the form $\{s_1,y_{s_2}\}$ for two distinct vertices $s_1,s_2\in S$,
which would imply that $s_1$ belongs to the maximal clique
$y_{s_2}$, contradicting the fact that $y_{s_2}\cap S = \{s_2\}$. Finally, the fact that $S$ is exactly the set of
$M\!$-saturated vertices in $X$ follows directly from the definition of $M$.

For the converse direction, assume that $S$ dominates $Y$ in $B_G$ and there exists an induced matching $M$ in $B_G$ such that $S$ is exactly the set of $M\!$-saturated vertices in $X$.
The fact that $S$ dominates $Y$ in $B_G$ implies that $S$ is a clique transversal in $G$. To see that $S$ is a minimal clique transversal, we will show that for every $s\in S$, the set $S\setminus \{s\}$ misses a maximal clique in $G$. Let $s\in S$. By the assumptions on $M$, vertex $s$ has a unique neighbor $y_s$ in $B_G$ such that $\{s,y_s\}$ is an edge of $M$. Furthermore, since $M$ is an induced matching in $B_G$,
vertex $y_s$ is not adjacent in $B_G$ to any vertex in $S\setminus \{s\}$. Thus, the set $S\setminus \{s\}$ misses the maximal clique $y_s$. We conclude that $S$ is a minimal clique transversal.
\end{proof}

The \emph{induced matching number} of a graph $G$ is the maximum size of an induced matching in~$G$.
\Cref{lem:clique-transversals-via-B_G} immediately implies the following.

\begin{corollary}\label{cor:bound}
For every graph $G$, the upper clique transversal number of $G$ is at most the induced matching number of $B_G$.
\end{corollary}

As another consequence of \Cref{lem:clique-transversals-via-B_G}, we obtain a sufficient condition for a set of vertices in a graph to be a minimal clique transversal of maximum size.

\begin{corollary}\label{cor:optimality}
Let $G$ be a graph, let $B_G = (X,Y;E)$ be its vertex-clique incidence graph, and let $S\subseteq V(G)$.
Suppose that $S$ dominates $Y$ in $B_G$ and there exists a maximum induced matching $M$ in $B_G$ such that $S$ is exactly the set of
$M\!$-saturated vertices in $X$. Then, $S$ is a minimal clique transversal in $G$ of maximum size.
\end{corollary}

To apply \Cref{cor:optimality} to proper interval graphs, we first state several characterizations of proper interval graphs in terms of their vertex-clique incidence graphs, establishing in particular a connection with bipartite permutation graphs.

\subsection{Characterizing proper interval graphs via their vertex-clique incidence graphs}

We first recall some concepts and results from the literature.
A bipartite graph $G = (X,Y;E)$ is said to be \emph{biconvex} if there exists a \emph{biconvex ordering} of (the vertex set of) $G$, that is, a pair $(<_X,<_Y)$ where $<_X$ is a linear ordering of $X$ and  $<_Y$ is a linear ordering of $Y$ such that for every $x\in X$, the vertices in $Y$ adjacent to $x$ appear consecutively with respect to the ordering $<_Y$, and, similarly, for every $y\in Y$, the vertices in $X$ adjacent to $y$ appear consecutively with respect to the ordering $<_X$.

We will need the following property of biconvex graphs.
Let $(<_X,<_Y)$ be a biconvex ordering of a biconvex graph $G = (X,Y;E)$.
Two edges $e$ and $f$ of $G$ are said to \emph{cross} (each other) if
there exist vertices $x_1,x_2\in X$ and $y_1,y_2\in Y$ such that
$\{e,f\} = \{\{x_1,y_2\},\{x_2,y_1\}\}$, $x_1<_X x_2$, and $y_1<_Y y_2$.
A biconvex ordering $(<_X,<_Y)$ of a biconvex graph $G = (X,Y;E)$ is said to be
\emph{induced-crossing-free} if for any two crossing edges $e = \{x_1,y_2\}$ and $f = \{x_2,y_1\}$,
either $x_1$ is adjacent to $y_1$ or $x_2$ is adjacent to $y_2$.

\begin{theorem}[Abbas and Stewart~\cite{MR1762199}]\label{thm:induced-crossing-free}
	Every biconvex graph has an induced-crossing-free biconvex ordering.
\end{theorem}

Given a bipartite graph $G = (X,Y;E)$, a \emph{strongly induced-crossing-free ordering} (or simply a \emph{strong ordering}) of $G$ is a pair $(<_X,<_Y)$ of linear orderings of $X$ and $Y$ such that for any two crossing edges $e = \{x_1,y_2\}$ and $f = \{x_2,y_1\}$, vertex $x_1$ is adjacent to $y_1$ and vertex $x_2$ is adjacent to $y_2$.

A \emph{permutation graph} is a graph $G=(V,E)$ that admits a permutation model, that is, vertices of $G$ can be ordered $v_1,\ldots, v_n$ such that there exists a permutation $(a_1,\ldots, a_n)$ of the set $\{1,\ldots, n\}$ such that for all $1\le i< j\le n$, vertices $v_i$ and $v_j$ are adjacent in $G$ if and only if $a_i>a_j$.
A \emph{bipartite permutation graph} is a graph that is both a bipartite graph and a permutation graph.

Spinrad, Brandst\"{a}dt, and Stewart~\cite{MR917130} characterized bipartite permutation graphs as exactly those bipartite graphs that have a strong ordering.
Moreover, observe that for a bipartite graph containing no isolated vertices, every strong ordering is also a biconvex ordering. 
We thus obtain the following characterization of bipartite permutation graphs.

\begin{theorem}
\label{thm:strongly-induced-crossing-free}
The following statements are equivalent for a bipartite graph $G = (X,Y;E)$:
\begin{enumerate}
\item $G$ is a bipartite permutation graph.
\item $G$ has a strong ordering.
\item $G$ has a strong biconvex ordering.
\end{enumerate}
\end{theorem}

\begin{sloppypar}
\Cref{thm:strongly-induced-crossing-free} implies the following property of bipartite permutation graphs equipped with a strong ordering.
\end{sloppypar}

\begin{sloppypar}
	\begin{corollary}\label{cor:induced-matching-no-crossing}
		Let $G = (X,Y;E)$ be a bipartite permutation graph, let $(<_X,<_Y)$ be a strong ordering of $G$, and let $M$ be an induced matching in $G$.
		Then, no two edges in $M$ cross.
	\end{corollary}
\end{sloppypar}

We will also use the following well-known characterization of proper interval graphs (see, e.g., Gardi~\cite{MR2364171}).

\begin{theorem}\label{thm:proper-interval}
A graph $G$ is a proper interval graph if and only if there exists an ordering $\sigma= (v_1,\ldots, v_n)$ of the vertices of $G$
and an ordering $\tau = (C_1,\ldots, C_k)$ of the maximal cliques of $G$ such that
for each $i\in \{1,\ldots, n\}$ the maximal cliques containing vertex $v_i$
appear consecutively in the ordering $\tau$, and
for each $j\in \{1,\ldots, k\}$ clique $C_j$ consists of consecutive vertices with respect to
ordering $\sigma$.
\end{theorem}

The following theorem gives several characterizations of proper interval graphs in terms of their vertex-clique incidence graphs.

\begin{theorem}\label{thm:pig-characterizations-via-BG}
Let $G$ be a graph. Then, the following statements are equivalent:
\begin{enumerate}
\item\label{statement-1} $G$ is a proper interval graph.
\item\label{statement-2} $B_G$ is a biconvex graph.
\item\label{statement-4} $B_G$ is a bipartite permutation graph.
\item\label{statement-5} $B_G$ has a strong ordering.
\item\label{statement-6} $B_G$ has a strong biconvex ordering.
\item\label{statement-3} $B_G$ has an induced-crossing-free biconvex ordering.
\end{enumerate}
\end{theorem}

\begin{proof}
Theorem~\ref{thm:proper-interval} implies the equivalence between statements~\ref{statement-1} and~\ref{statement-2}.
Equivalence between statements~\ref{statement-2} and~\ref{statement-3} follows from Theorem~\ref{thm:induced-crossing-free}.
Equivalence among statements~\ref{statement-4}, \ref{statement-5}, and~\ref{statement-6} follows from Theorem~\ref{thm:strongly-induced-crossing-free}.
Clearly, statement~\ref{statement-6} implies statement~\ref{statement-2}.

Finally, we show that statement~\ref{statement-1} implies statement~\ref{statement-5}.
Fix a proper interval representation $(I_v:v\in V(G))$ of $G$. Let $B_G = (X,Y;E)$ where $X = V(G)$ and $Y = \mathcal{C}_G$.
Let $<_X$ be the ordering of $X$ corresponding to the left-endpoint order of the intervals. (Note that since no interval properly contains another, the left-endpoint order and the right-endpoint order are the same.)
As shown in~\cite{MR2364171}, every maximal clique $C\in \mathcal{C}_G (= Y)$ consists of consecutive vertices with respect to $<_X$.
Since the cliques are maximal, no two cliques in $Y$ have the same first vertex with respect to $<_X$, hence there is a unique and well defined ordering $<_Y$ of $Y$ that orders the cliques in increasing order of their first vertices in the vertex order.
We claim that the pair $(<_X, <_Y)$ is a strong ordering of $B_G$. Consider any two crossing edges $e = \{x_1,y_2\}$ and $f = \{x_2,y_1\}$.
We may assume that $x_1 <_X x_2$ and $y_1<_Y y_2$.
Since $y_1<_Y y_2$, we have $s_1<_X s_2$, where $s_i$ is the first vertex of $y_i$ for $i\in \{1,2\}$.
Furthermore, since $x_1$ and $y_2$ are adjacent in $B_G$, vertex $x_1$ belongs to $y_2$, and thus
$s_2\le_X x_1$. Consequently, $s_1<_X  s_2 \le_X x_1 <_X x_2$. Thus, since $x_2$ belongs to $y_1$, also $x_1$ belongs to $y_1$. This implies that
$x_1$ and $y_1$ are adjacent in $B_G$.
Finally, since $y_1<_Y y_2$, clique $y_2$ ends strictly after clique $y_1$, and since $x_2$ belongs to $y_1$, we conclude that $x_2$ also belongs to
$y_2$. Thus, $x_2$ and $y_2$ are adjacent in $B_G$.
It follows that the pair $(<_X, <_Y)$ is a strong ordering of $B_G$, as claimed.
\end{proof}

\subsection{Maximum induced matchings in bipartite permutation graphs, revisited}

Our goal is to show that if $G$ is a proper interval graph, then the sufficient condition given by \Cref{cor:optimality} is satisfied, namely, there exists a maximum induced matching $M$ in $B_G$ such that the set $S$ of $M\!$-saturated vertices in $X$ dominates $Y$ in $B_G$.
By \Cref{cor:optimality}, this will imply $\tau_c^+(G) = |S| = |M|$.

We show the claimed property of $B_G$ as follows.
First, by applying \Cref{thm:pig-characterizations-via-BG}, we infer that the graph $B_G$ is a bipartite permutation graph.
Second, by construction, $B_G$ does not have any isolated vertices and no two distinct vertices in $Y$ have comparable neighborhoods in $X$.
It turns out that these properties are already enough to guarantee the desired conclusion.
We show this by a careful analysis of a linear-time algorithm due to Chang from~\cite{MR2024264} for computing a maximum induced matching in bipartite permutation graphs.

Chang's algorithm is based on a strong biconvex ordering $(<_X,<_Y)$ of a connected bipartite permutation graph $G = (X,Y;E)$.
Let $x_1,\ldots,x_s$ be the ordering of $X$ such that for all $i,j\in \{1,\ldots, s\}$, we have $i<j$ if and only if $x_i<_X x_j$.
Similarly, let $y_1,\ldots,y_t$ be the ordering of $Y$ such that for all $i,j\in \{1,\ldots, t\}$, we have $i<j$ if and only if $y_i<_Y y_j$.
For each vertex $v\in X$, let $\min(v)$ and $\max(v)$ denote the smallest and the largest $i$ such that $y_i$ is adjacent to $v$, respectively; for vertices in $Y$, $\min(v)$ and $\max(v)$ are defined similarly.
Using this notation, the pseudocode of Chang's algorithm is given as Algorithm~\ref{alg:BPG}.

\begin{algorithm}[h!]
	\caption{Computing a maximum induced matching of a connected bipartite permutation graph}\label{alg:BPG}
	
	\KwIn{A connected bipartite permutation graph $G = (X,Y;E)$ with $E\neq\emptyset$.}
	\KwOut{A maximum induced matching $M$ of $G$.}
	\BlankLine
	
	compute a strong biconvex ordering $(<_X,<_Y)$ of $B_G$\;
	
	compute the values $\min(v)$ and $\max(v)$ for all $v\in V(G)$\;
	
    $M \leftarrow \{\{x_s,y_t\}\}$\; \tcp{the vertices $x_s$ and $y_t$ are adjacent in $G$}

    let $i = s$ and $j = t$\;

	\While{$\min(x_i)\neq 1$ and $\min(y_j)\neq 1$} { \nllabel{line:iteration}
		
		let $p = \min(y_j)$ and $q = \min(x_i)$\; \tcp{note that $p\ge 2$ and $q\ge 2$}
		
		\If{$\min(x_p)<q$ and $\min(y_q)<p$}
		{
		    $M\leftarrow M\cup \{\{x_{p-1},y_{q-1}\}\}$\; 		$i \leftarrow p-1$\;
		    $j \leftarrow q-1$\;
		}
		\If{$\min(x_p)=q$ and $\min(y_q)<p$}
		{
		    $M\leftarrow M\cup \{\{x_{\max(y_{q-1})},y_{q-1}\}\}$\;
		    $i \leftarrow \max(y_{q-1})$\;
		    $j \leftarrow q-1$\;
		}
		\If{$\min(x_p)<q$ and $\min(y_q)=p$}
		{
		    $M\leftarrow M\cup \{\{x_{p-1},y_{\max(x_{p-1})}\}\}$\;
		    $i \leftarrow p-1$\;
		    $j \leftarrow \max(x_{p-1})$\;
		}
		\tcp{exactly one the of above three {\bf if} statements is true}
	}
	
	\Return{$M$}\; \nllabel{line:returnM}
\end{algorithm}

\begin{theorem}[Chang~\cite{MR2024264}]\label{thm:Chang}
Algorithm~\ref{alg:BPG} computes a maximum induced matching in a given connected bipartite permutation graph with at least one edge and runs in linear time.
\end{theorem}

The linear time complexity of Algorithm~\ref{alg:BPG} relies on a result due to Spinrad, Brandst\"{a}dt, and Stewart~\cite{MR917130} and  Sprague~\cite{MR1369371} stating that a strong biconvex ordering of a given bipartite permutation graph can be computed in linear time.

\begin{theorem}\label{thm:MIM-bip-perm}
Given a bipartite permutation graph $G = (X,Y;E)$, there is a linear-time algorithm that computes a maximum induced matching $M$ in $G$ such that, if $G$ has no isolated vertices and no two vertices in $Y$ have comparable neighborhoods in $G$, then the set of $M\!$-saturated vertices in $X$ dominates $Y$.
\end{theorem}

\begin{proof}
Let $G = (X,Y;E)$ be a bipartite permutation graph.
We consider two cases.
First, assume first that $G$ either contains an isolated vertex or two vertices in $Y$ with comparable neighborhoods in $G$.
In this case, it suffices to show that  there is a linear-time algorithm that computes a maximum induced matching in $G$.
We may assume without loss of generality that $G$ is connected; otherwise, we compute in linear time the connected components of $G$ using breadth-first search, solve the problem on each component, and combine the solutions.
Assuming $G$ is connected, we compute a maximum induced matching in $G$ in linear time using \Cref{thm:Chang}.

The second case is when $G$ has no isolated vertices and no two vertices $y,y'\in Y$ have comparable neighborhoods in $G$, that is, $N(y)\subseteq N(y')$ or $N(y')\subseteq N(y)$, if and only if $y = y'$.
Again, we first argue that it suffices to consider the case of connected graphs.
In the general case, we proceed as follows.
First, the connected components of $G$ can be computed in linear time using breadth-first search.
Second, since no two vertices in $Y$ have comparable neighborhoods in $G$, the same is also true for each connected component.
Third, assume that each connected component $C = (X_C,Y_C;E_C)$ has a maximum induced matching $M_C$ such that, if no two vertices in $Y_C$ have comparable neighborhoods in $G$, then the set of $M_C$-saturated vertices in $X_C$ dominates $Y_C$.
Thus, the union of all such maximum induced matchings $M_C$ yields a maximum induced matching $M$ in $G$ such that the set of $M\!$-saturated vertices in $X$ dominates~$Y$.

Assume now that $G$ is connected.
By \Cref{thm:Chang}, Algorithm~\ref{alg:BPG} computes in linear time a maximum induced matching $M$ in $G$. 
We show that the set of $M\!$-saturated vertices in $X$ dominates $Y$.
Suppose for a contradiction that there exists a vertex $y\in Y$ that is not adjacent to any $M$-saturated vertex in $X$.
Clearly, $y$ is not an endpoint of a matching edge.
By construction, no two edges of $M$ cross.
Thus, we may order the edges of $M$ linearly as
$M = \{\{x_{i_1},y_{j_1}\},\ldots, \{x_{i_r},y_{j_r}\}\}$
so that $i_1<\dots<i_r = s$ and  $j_1<\dots<j_r = t$.
Note that the algorithm added the edges to $M$ in the order
$\{x_{i_r},y_{j_r}\}, \{x_{i_{r-1}},y_{j_{r-1}}\}, \ldots, \{x_{i_1}, y_{j_1}\}$.
Since $i_r = s$ and $j_r = t$, there exists a smallest integer $k\in \{1,\ldots, r\}$ such that $y<_Y y_{j_k}$.
Furthermore, since no two vertices in $Y$ have comparable neighborhoods, there exists a vertex $x\in X$ adjacent to $y$ but not to $y_{j_k}$.
The edge $\{x_{i_k},y_{j_k}\}$ belongs to the matching $M$, and hence the vertex $x_{i_k}$ is adjacent to $y_{j_k}$ but not to $y$, since no neighbor of $y$ is $M$-saturated.
Next, observe that $x<_X x_{i_k}$, since otherwise the presence of the edges $\{x_{i_k},y_{j_k}\}$ and $\{x,y\}$ would imply, using the fact that $(<_X,<_Y)$ is a strong ordering of $G$, that $x_{i_k}$ is adjacent to~$y$.

Consider the iteration of the {\bf while} loop of the algorithm right after the edge $\{x_{i_k},y_{j_k}\}$ was added to $M$.
Then $i = i_k$ and $j = j_k$ at the beginning of that loop.
Since $x<_Xx_{i}$ and $y<_Y y_i$, the facts that $x_i$ and $y_j$ are non-adjacent to $y$ and $x$, respectively, and that $(<_X,<_Y)$ is a strong ordering of $G$, imply that the condition $\min(x_i)\neq 1$ and $\min(y_j)\neq 1$ of the {\bf while} loop is satisfied.
Hence, the algorithm enters the {\bf while} loop.
Let $p = \min(y_j)$ and $q = \min(x_i)$.
Using the fact that $(<_X,<_Y)$ is a strong ordering of $G$, we infer that $x<_Xx_p$ and $y<_Yy_q$.
Since the graph $G$ is connected, exactly one of the conditions of the three {\bf if} statements within the {\bf while} loop will be satisfied and the algorithm adds at least one more edge $e = \{x_{i_{k-1}},y_{j_{k-1}}\}$ to $M$.
In particular, $(i_{k-1},j_{k-1})\in \{(p-1,q-1),(\max(y_{q-1}),q-1),(p-1,\max(x_{p-1}))\}$.
By the definition of $k$, we have $y_{j_{k-1}}<_Yy$.
Since we also have $y<_Yy_q$, we infer that $j_{k-1}<q-1$ and therefore $j_{k-1}= \max(x_{p-1})$ and consequently $i_{k-1} = p-1$.
The vertex $x_{p-1} = x_{i_{k-1}}$ is an endpoint of an edge in $M$ and therefore not adjacent to $y$, since no neighbor of $y$ is $M$-saturated.
In particular, $x_{p-1}\neq x$ and thus $x<_Xx_p$ implies that $x<_Xx_{p-1}$.
But now, the presence of the edges $\{x,y\}$ and $\{x_{p-1},y_{j_{k-1}}\}$ together with $x<_Xx_{p-1}$, $y_{j_{k-1}}<_Yy$, and the fact that $(<_X,<_Y)$ is a strong ordering of $G$, implies that $x_{p-1}$ is adjacent to $y$, a contradiction.
\end{proof}

\subsection{Solving UCT in proper interval graphs in linear time}

We now have everything ready to prove the announced result.

\begin{sloppypar}
\begin{theorem}
{\sc Upper Clique Transversal} can be solved in linear time in the class of proper interval graphs.
\end{theorem}
\end{sloppypar}

\begin{proof}
The algorithm proceeds in three steps.
In the first step, we compute from the input graph $G=(V,E)$
its vertex-clique incidence graph $B_G$, with parts $X = V$ and $Y = {\mathcal C}_G$.
By Theorem~\ref{thm:pig-characterizations-via-BG}, the graph $B_G$ is a bipartite permutation graph.
In the second step of the algorithm, we compute a maximum induced matching $M$ of $B_G$, using \Cref{thm:MIM-bip-perm}.
Finally, the algorithm returns the set of $M$-saturated vertices in $X$.
The pseudocode is given as Algorithm~\ref{alg:PIG}.

\begin{algorithm}[h!]
	\caption{Computing a maximum minimal clique transversal of a proper interval graph}\label{alg:PIG}
	
	\KwIn{A proper interval graph $G = (V,E)$.}
	\KwOut{A maximum minimal clique transversal of $G$.}
	\BlankLine
	
	compute the vertex-clique incidence graph $B_G$, with parts $X = V$ and $Y = {\mathcal C}_G$\; \nllabel{line:BG}

	compute a maximum induced matching $M$ of $B_G$;
 \nllabel{line:mim}
	
	compute the set $M_X$ of $M$-saturated vertices in $X$\;
	\Return{$M_X$}\; \nllabel{line:return}
\end{algorithm}

\textit{Correctness.}
By construction, the set $M_X$ returned by the algorithm is a subset of $X$, and thus a set of vertices of $G$.
Since every vertex of $G$ belongs to a maximal clique, and every maximal clique contains a vertex, $B_G$ does not have any isolated vertices.
Furthermore, since the vertices of $Y$ are precisely the maximal cliques of $G$, no two vertices in $Y$ have comparable neighborhoods in $B_G$.
Therefore, by \Cref{thm:MIM-bip-perm}, the set $M_X$ dominates $Y$.
By Corollary~\ref{cor:optimality}, $M_X$ is a maximum minimal clique transversal in $G$.

\medskip
\textit{Time complexity.}
By~\Cref{thm:chordal-maximal-cliques} and the fact that every proper interval graph is a chordal graph, $G$ has at most $|V|$ maximal cliques, which can be computed in linear time.
In particular, this implies that the vertex-clique incidence graph $B_G$ has at most $2|V|$ vertices and $\mathcal{O}(|V|+|E|)$ edges, and can be computed in linear time.
Since $B_G$ is a bipartite permutation graph, a maximum induced matching of $B_G$ can be computed in linear time, see Theorem~\ref{thm:MIM-bip-perm}.
The set of $M$-saturated vertices in $X$ can also be computed in linear time.
Thus, the overall time complexity of the algorithm is $\mathcal{O}(|V|+|E|)$.
\end{proof}

The above proof also shows the following.

\begin{theorem}\label{thm:PIG-uctn=imn}
For every proper interval graph $G$, the upper clique transversal number of $G$ is equal to the induced matching number of $B_G$.
\end{theorem}

We conclude the section by showing that the result of \Cref{thm:PIG-uctn=imn} does not generalize to the class of interval graphs.

\begin{observation}\label{obs:interval-graphs-gap}
There exist interval graphs such that the difference between the induced matching number of their vertex-clique incidence graph and the upper clique transversal number of the graph is arbitrarily large.
\end{observation}

\begin{proof}
Let $q\ge 2$ and let $G$ be the graph obtained from two disjoint
copies of the star graph $K_{1,q}$ by adding an edge between the two vertices of degree $q$.
It is easy to see that $G$ is an interval graph.

We claim that the upper clique transversal number of $G$ is at most $q+1$, while the induced matching number of $B_G$ is at least $2q$.
To see that the upper clique transversal number of $G$ is at most $q+1$, consider an arbitrary minimal
clique transversal $S$ of $G$. Then $S$ must contain at least one of the vertices of degree $q+1$; let
$u$ be such a vertex. Then, since $S$ is minimal, it cannot contain any of the $q$ neighbors of $u$ that are of degree $1$ in $G$.
Thus, $S$ either consists of the two vertices of degree $q+1$ in $G$, or contains $u$ and all its non-neighbors in $G$.
In either case, $S$ is of size at most $q+1$.

It remains to show that the induced matching number of the vertex-clique incidence graph of $G$ is at least $2q$.
As usual, let $B_G = (X,Y;E)$, with $X = V(G)$ and $Y = \mathcal{C}_G$.
Since $G$ is triangle-free and has no isolated vertices, the maximal cliques of $G$ are exactly the edges of $G$, and
the edges of $B_G$ are the pairs $\{x,e\}$ where $x\in V(G)$, $e\in E(G)$, and $x$ is an endpoint of $e$.
Thus, $B_G$ is isomorphic to the graph obtained from $G$ by subdividing each edge.
Let $M$ be the set of edges of $B_G$ of the form $\{x,e\}$ where $x$ is a vertex in $B_G$ of degree $1$ and $e$ is the unique edge incident with it.
Then $M$ is an induced matching in $B_G$ of size $2q$ and hence the induced matching number of $B_G$ is at least $2q$.
\end{proof}

\section{A linear-time algorithm for UCT for cographs}\label{sec:cographs}

In this section, we discuss UCT in the class of cographs and apply results from the literature to obtain a linear-time algorithm.
The class of cographs is defined recursively as follows.
\begin{itemize}
    \item The one-vertex graph $K_1$ is a cograph.
\item Given two cographs $G_1$ and $G_2$, their disjoint union $G_1+G_2$ is a cograph.
\item Given two cographs $G_1$ and $G_2$, their join  $G_1\ast G_2$ (that is, the graph obtained from the disjoint union of $G_1$ and $G_2$ by adding all the edges having one endpoint in $G_1$ and the other one in $G_2$) is a cograph.
\item There are no other cographs.
\end{itemize}

The class of cographs has many equivalent characterizations (see, e.g.,~\cite{MR619603,MR0441560}); in particular, a graph is a cograph if and only if it is $P_4$-free.
As shown by Gurvich~\cite{MR0441560} and also by Karchmer, Linial, Newman, Saks, Wigderson~\cite{MR1217758} (see also Gurvich~\cite{MR2755907} and Golumbic and Gurvich~\cite[Chapter 10]{MR2742439}), if $G$ is a $P_4$-free graph, then the minimal clique transversals of $G$ are exactly its maximal independent sets.
This implies, in particular, that every cograph $G$ satisfies $\tau_c^+(G) = \alpha(G)$ and that the problem of computing an upper clique transversal in a given cograph $G$ is equivalent to the problem of computing a maximum independent set in $G$.
This problem is known to be solvable in linear time in the class of cographs; see McConnell and Spinrad~\cite{zbMATH01334546} for a linear-time algorithm for maximum independent set problem in the more general class of cocomparability graphs.
We therefore obtain the following result.

\begin{theorem}\label{thm:cographs}
{\sc Upper Clique Transversal} can be solved in linear time in the class of cographs.
\end{theorem}

Note also that while, by~\Cref{thm:MMCT-split}, every split graph $G$ satisfies $\tau_c^+(G)\le \alpha(G)$, in the class of cographs this inequality is satisfied with equality.

For a direct proof of \Cref{thm:cographs} based on the recursive structure of cographs, see~\cite{MU2024arXiv}.

\section{UCT for graphs with bounded cliquewidth}\label{sec:bounded-cliquewidth}

In this section, we show that UCT can be solved in polynomial time in any class of graphs with bounded cliquewidth.
The term ``cliquewidth'' was introduced in 2000 by Courcelle and Olariu~\cite{MR1743732}, although the concept has been defined earlier, in the context of graph grammars in 1993 by Courcelle, Engelfriet, and Rozenberg~\cite{MR1217156}.
Cliquewidth is a graph complexity measure that is bounded whenever treewidth is bounded (see~\cite{MR1743732,MR2148860}) but, unlike treewidth, can also be bounded on classes of dense graphs, such as complete graphs and complete bipartite graphs.
The \emph{cliquewidth} of a graph $G$ is defined as the minimum number $k$ such that $G$ admits a \emph{$k$-expression}, that is, a construction of a graph isomorphic to $G$ in which each vertex is equipped with a label $\ell(v)$ from the set $L = \{1,\ldots, k\}$ of labels using the following operations:
\begin{enumerate}[(1)]
    \item Creation of a new graph with a single vertex $v$ having label $i\in L$. (This operation is denoted by $i_v$.)
    \item Disjoint union $G_1\oplus G_2$ of two already constructed labeled graphs $G_1$ and $G_2$.
\item For any two distinct labels $i,j\in L$, the addition of all edges between every vertex with label $i$ and every vertex with label $j$ (denoted by $\eta_{i,j}$).
\item For any two distinct labels $i,j\in L$, relabeling of every vertex with label $i$ to have label $j$  (denoted by $\rho_{i\to j}$).
\end{enumerate}

For any positive integer $k$, many {\sf NP}-hard decision and optimization problems on graphs can be solved in linear time for a graph $G$ given with a $k$-expression.
In particular, as shown by Courcelle, Makowsky, and Rotics in~\cite{MR1739644}, this is the case for any graph problem that can be defined in {\sf MSO$_1$}, the fragment of monadic second order logic where quantified relation symbols are permitted on relations of arity 1 (such as vertices), but not of arity 2 (such as edges) or more, which means that, with graphs, one can quantify over sets of vertices.
Furthermore, as shown by Fomin and Korhonen~\cite{MR4783063}, for every fixed positive integer $k$ there is an algorithm running in time $2^{2^{\mathcal{O}(k)}}n^2$ that takes as input an $n$-vertex graph with cliquewidth at most $k$ and computes a $(2^{2k+1}-1)$-expression of $G$.
Cliquewidth is closely related to some other graph width parameters, in particular to rankwidth~\cite{MR2232389} and Boolean-width~\cite{MR2857670,MR3126918}, since bounded cliquewidth is equivalent to bounded rankwidth or bounded Boolean-width.
For further information about cliquewidth and related width parameters, we refer to the surveys~\cite{MR3967291,DBLP:journals/cj/HlinenyOSG08}.

The recursive structure of cographs implies that cographs have cliquewidth at most~$2$.
Thus, the next theorem is a generalization of \Cref{thm:cographs}.

\begin{theorem}\label{thm:bdd-cwd}
For every positive integer $k$, {\sc Upper Clique Transversal} can be solved in linear time in the class of graphs of cliquewidth at most $k$ if the input graph is given with a $k$-expression.
\end{theorem}

\begin{proof}
It suffices to show that {\sc Upper Clique Transversal} can be defined in {\sf MSO$_1$}, as the theorem will then follow from the result of Courcelle, Makowsky, and Rotics (\cite[Theorem 4]{MR1739644}).

To show that {\sc Upper Clique Transversal} can be defined in {\sf MSO$_1$}, we construct a fixed {\sf MSO$_1$}-formula $\varphi$ such that for any graph $G$ and a set $X\subseteq V(G)$, $(G,X)\models \varphi$ if and only if $X$ is a minimal clique transversal in $G$.
The graph $G$ is represented with the universe $V$ (the set of vertices), with one variable per vertex, and the binary adjacency relation $E$.
Set variables are represented with capital letters, and the membership relation $v\in X$ is written as a unary relation $X(v)$.
The formula $\varphi$ is composed from the following simpler formulas, as follows.
\begin{itemize}
    \item First, we have a formula $\varphi_1$ such that $(G,X)\models \varphi_1$ if and only if $X$ is a clique in $G$, that is, any two distinct vertices in $X$ are adjacent.
    In formulae: \[\varphi_1(G,X) \equiv (\forall x)(\forall y)(x\neq y \wedge X(x) \wedge X(y) \Rightarrow E(x,y))\,.\]
    \item Next, formula $\varphi_2$ is such that $(G,X)\models \varphi_2$ if and only if $X$ is a maximal clique in $G$, that is, $X$ is a clique and every vertex not in $X$ is nonadjacent to some vertex in $X$.
     In formulae:
     \[\varphi_2(G,X) \equiv\varphi_1(G,X) \wedge (\forall x)(\neg X(x) \Rightarrow (\exists y)(X(y) \wedge \neg E(x,y))\,.\]
    \item Next, formula $\varphi_3$ is such that $(G,X)\models \varphi_3$ if and only if $X$ is clique transversal in $G$, that is, $X$ contains a vertex of every maximal clique in $G$.
    In formulae:
     \[\varphi_3(G,X)\equiv (\forall Y)(
    \varphi_2(G,Y) \Rightarrow (\exists x)(X(x)\wedge Y(x))\,.\]
    \item Finally, we have the desired formula $\varphi$ such that $(G,X)\models \varphi$ if and only if $X$ is minimal clique transversal in $G$, that is, $X$ is clique transversal in $G$ and for every $x\in X$, the set $X\setminus \{x\}$ is not a clique transversal in $G$.
    In formulae:
     \[\varphi(G,X)\equiv \varphi_3(G,X) \wedge (\forall x)(X(x)\Rightarrow (\neg \varphi_3(G,X\setminus\{x\})))\,.\]
\end{itemize}
This completes the proof.
\end{proof}

Since every problem that can be defined in {\sf MSO$_1$} can also be expressed in the more general logic {\sf MSO$_2$} (where one can quantify over sets of vertices as well as sets of edges), a result by Arnborg, Lagergren, and Seese~\cite{MR1105479} applies, stating that in any class of graphs with bounded treewidth, any optimization problem expressible in {\sf MSO$_2$} can be solved in linear time.
The result assumes that the graph is equipped with a tree decomposition of bounded width; however, as shown by Bodlaender~\cite{MR1417901}, such a tree decomposition can be computed in linear time.
Therefore, {\sc Upper Clique Transversal} can be solved in linear time in any class of graphs with bounded treewidth (which is not surprising, given \Cref{thm:bdd-cwd} and the fact that bounded treewidth implies bounded cliquewidth).

\section{Conclusion}\label{sec:conclusion}

We performed a systematic study of the complexity of \textsc{Upper Clique Transversal} in various graph classes, showing, on the one hand, {\sf NP}-completeness of the problem in the classes of chordal graphs, chordal bipartite graphs, and line graphs of bipartite graphs, and, on the other hand, linear-time solvability in the classes of split graphs and proper interval graphs.

Our work leaves open several questions.

\begin{question}\label{question1}
What is the complexity of computing a minimal clique transversal in a given graph?
\end{question}

\begin{sloppypar}
UCT can be solved in polynomial time in classes of graphs with bounded cliquewidth.
It is an interesting question whether similar results can be derived for other width parameters  generalizing treewidth, such as tree-independence number~\cite{MR3775804,MR4664382}, mim-width~\cite{vatshelle2012new}, their common generalization sim-width~\cite{MR3721445}, and twin-width~\cite{MR4402362}.
Our {\sf NP}-completeness result for UCT in chordal graphs (\Cref{sec:NP-c-chordal}) implies that UCT is {\sf NP}-hard for graphs with tree-independence number at most one (see~\cite{MR4664382}) and consequently for graphs with sim-width at most one (see~\cite[Lemma 5]{MR4563598}).
Similarly, the {\sf NP}-completeness result for UCT in bipartite planar graphs (\Cref{thm:NP-c-chordal-bip}) implies that UCT is {\sf NP}-hard for graphs with twin-width at most~$6$ (see~\cite{MR4970191}).
However, the following question remains open.
\end{sloppypar}

\begin{question}
What is the complexity of \textsc{Upper Clique Transversal} in graph classes with bounded mim-width?
\end{question}

Let us note, however, that due to the connection with \textsc{Independent Dominating Set} (cf.~the proof of \Cref{thm:NP-c-chordal-bip}) and known algorithmic results for graphs of bounded mim-width (see~\cite{MR3126918,MR3126917}), UCT is polynomial-time solvable in any class of triangle-free graphs in which mim-width is bounded and quickly computable, for example, for circular convex graphs (see~\cite{MR4668325}).

By \Cref{cor:bound}, the upper clique transversal number of any graph $G$ is bounded from above by the induced matching number of its vertex-clique incidence graph $B_G$.
\Cref{thm:PIG-uctn=imn} shows that this upper bound is attained with equality if $G$ is a proper interval graph.
This motivates the following.

\begin{question}
For what graphs $G$ is the upper clique transversal number equal to the induced matching number of the vertex-clique incidence graph?
\end{question}

While not all interval graphs have the stated property (by \Cref{obs:interval-graphs-gap}), the property is satisfied by graphs other than proper interval graphs; for example, all cycles have the property.

The upper clique transversal number is a trivial upper bound for the clique transversal number; however, the ratio between these two parameters can be arbitrarily large in general.
For instance, in the complete bipartite graph $K_{1,q}$ the former one has value $q$ while the latter one has value $1$.
This leads to the following.

\begin{question}
For which graph classes is the ratio (or even the difference) between the clique transversal number and the upper clique transversal number bounded?
\end{question}

\begin{sloppypar}
The focus of our paper was on classical complexity, in the sense that the aim was to understand which restrictions on the input graphs result in polynomially solvable cases.
It would be natural to explore the complexity of the UCT problem also in terms of other measures, for example with respect to parameterized complexity and approximability.

Regarding parameterized complexity, \Cref{thm:bdd-cwd} leads to an FPT algorithm for the UCT problem when parameterized by the cliquewidth of the graph.
For the natural parameterization of the problem, that is, when the parameter is the upper clique transversal number, it was shown using hypergraph techniques that the problem is in {\sf XP}; see~\cite{boros2023dually}.
Furthermore, the problem was recently shown to be {\sf W[1]}-hard; see~\cite{gollin2026w1hardnessuppercliquetransversal}.
It may be interesting to study the question with respect to other parameters.
\end{sloppypar}

\begin{question}
How well can the upper clique transversal number of a graph be approximated in polynomial time?
\end{question}

Note that this question is related to \Cref{question1}.
Indeed, suppose that there exists a polynomial-time approximation algorithm for the maximization version of the UCT problem that is required to return not only the solution value but also a feasible solution.
Then, a minimal clique transversal in a given graph can be computed in polynomial time, providing an answer to \Cref{question1}.

\subsection*{Acknowledgements}

We are grateful to  Meike Hatzel, Ekkehard Köhler, Nikolaos Melissinos, and Haiko M\"uller for their helpful comments.
The work of the first named author is supported in part by the Slovenian Research and Innovation Agency (I0-0035, research program P1-0285 and research projects J1-60012, J1-70035, J1-70046, and N1-0370).
Part of the work was done while the author was visiting Osaka Prefecture University in Japan, under the operation Mobility of Slovene higher education teachers 2018--2021, co-financed by the Republic of Slovenia and the European Union under the European Social Fund.
The second named author is partially supported by JSPS KAKENHI Grant Number JP17K00017, JP20H05964, and JP21K11757, Japan.

\end{document}